\documentclass[12pt,twoside]{amsart}

\usepackage{charter, amsmath, amsbsy, amsfonts, amssymb, stmaryrd, amsthm, mathrsfs, graphicx, amscd, tikz-cd, cancel, array, tabulary, bm}
\usepackage{srcltx}
\usepackage[mathscr]{eucal}

\usepackage{xcolor}

\usepackage{fbb}

\usepackage[
	hypertexnames=false,
	hyperindex,
	pagebackref,
	pdftex,
	breaklinks=true,
	bookmarks=false,
	colorlinks,
	linkcolor=red,
	citecolor=blue,
	urlcolor=orange,
]{hyperref}

\usepackage[all]{xy}
\usepackage{latexsym,graphicx}

\usepackage{setspace}

\def\AA{{\mathbb A}}

\def\CC{{\mathbb C}}

\def\GG{{\mathbb G}}

\def\NN{{\mathbb N}}

\def\QQ{{\mathbb Q}} 
\def\RR{{\mathbb R}} 
\def\SS{{\mathbb S}}

\def\UU{{\mathbb U}} 
\def\VV{{\mathbb V}} 
 
\def\ZZ{{\mathbb Z}} 

\def\Acal{{\mathcal A}}

\def\Dcal{{\mathcal D}}
 
\def\Fcal{{\mathcal F}}

\def\Hcal{{\mathcal H}}

\def\Kcal{{\mathcal K}}

\def\Ocal{{\mathcal O}}

\def\Vcal{{\mathcal V}} 
\def\Xcal{{\mathcal X}} 
 
\def\mfrak{{\mathfrak m}}
 
\def\gfrak{{\mathfrak g}}
\def\kfrak{{\mathfrak k}}

\def\mfrak{{\mathfrak m}}

\def\pfrak{{\mathfrak p}}

\def\Hscr{{\mathscr H}}

\def\Kscr{{\mathscr K}}
\def\Pscr{{\mathscr P}}
\def\Qscr{{\mathscr Q}}
\def\Sscr{{\mathscr S}}
 
\def\Abold{{\bm A}}

\def\Bbold{{\bm B}}

\def\Pbold{{\bm P}}

\def\Ebold{{\pmb E}}
\def\Fbold{{\pmb F}}
\def\Gbold{{\pmb G}}
\def\Hbold{{\pmb H}}
\def\Tbold{{\pmb T}}
\def\Lbold{{\bm L}}

\def\G{{\Gamma}}
\def\g{{\gamma}} 

\def\Ztilde{\tilde{Z}}

\def\psitilde{\tilde{\psi}}

\def\Hod{{\rm Hod}}
\def\Sh{{\rm Sh}}

\def\bs{\backslash}

\def\pt{{\scriptscriptstyle\bullet}}

\newcommand\supp{\operatorname{supp}}
\newcommand\Lie{\operatorname{Lie}}
\newcommand\ad{\operatorname{ad}}
\newcommand\Ad{\operatorname{Ad}}
\newcommand\Int{\operatorname{Int}}

\newcommand\Res{\operatorname{Res}}
\newcommand\GL{\pmb{\operatorname{GL}}}

\newcommand\Zcent{\operatorname{\bf Z}}

\newcommand\uRad{\operatorname{R_u}}

\newcommand\Hom{\operatorname{Hom}}

\newcommand\im{\operatorname{Im}}

\newcommand\Nor{\operatorname{\bf N}}

\newcommand{\HL}{\textnormal{HL}}
 
\theoremstyle{plain}
\newtheorem{theorem}{Theorem}[section] 
\newtheorem{lemma}[theorem]{Lemma}
\newtheorem{sublemma}[theorem]{Sublemma}

\newtheorem{proposition}[theorem]{Proposition}

\newtheorem{corollary}[theorem]{Corollary}
\newtheorem{conjecture}[theorem]{Conjecture}

\theoremstyle{definition} 
\newtheorem{definition}[theorem]{Definition}

\theoremstyle{definition}

\theoremstyle{definition}

\theoremstyle{remark} 
\newtheorem{remark}[theorem]{Remark}
\newtheorem{question}[theorem]{Question}

\theoremstyle{remark}

\theoremstyle{remark}

\begin{document}

\title[Geometric André-Oort for VHS]{On the geometric André-Oort conjecture for variations of Hodge structures}

\author{\scshape Jiaming Chen}
\address{Universit\'e de Paris (Institut de math\'ematiques de Jussieu - Paris Rive Gauche, Paris)} 
\email{jiaming.chen@imj-prg.fr}

\begin{abstract}
Let $\VV$ be a polarized variation of integral Hodge structure on a smooth complex quasi-projective variety $S$.
In this paper, we show that the union of the \textit{non-factor} special subvarieties for $(S, \VV)$, which are of Shimura type with dominant period maps, is a finite union of special subvarieties of $S$. This generalizes previous results of Clozel and Ullmo \cite{ClozelUllmo05}, Ullmo \cite{Ullmo07} on the distribution of the non-factor (in particular, strongly) special subvarieties in a Shimura variety to the non-classical setting and also answers positively the geometric part of a conjecture of Klingler on the Andr\'e-Oort conjecture for variations of Hodge structures.
\end{abstract}

\maketitle

\section{Introduction}
\subsection{Motivation} The classical Andr\'e-Oort conjecture, which describes the distribution of CM points (points with complex multiplication) on a Shimura variety, asserts that the Zariski closure of a subset of CM points in a Shimura variety is special (namely, an irreducible component of a Hecke translate of a Shimura subvariety). It is the analog in a Hodge-theoretic context of the Manin-Mumford conjecture (a theorem of  Raynaud (\cite{Raynaud88}) stating that an irreducible subvariety of a complex abelian variety containing a Zariski-dense set of torsion points is the torsion translate of an abelian subvariety. It has recently been proved for the Shimura variety $\Acal_g$ moduli space of principally polarized complex abelian varieties of dimension $g$ (and more generally for mixed Shimura varieties whose pure part are of abelian type) following a strategy proposed by Pila and Zannier, through the work of many authors (\cite{PilaTsimerman14}, \cite{KUY16}, \cite{Gao17}, \cite{AGHM18}, \cite{YuanZhang18}, \cite {Tsimerman18}). Recently, Klingler (\cite{Klingler17}) formulated a generalization of the Andr\'e-Oort conjecture (in fact, of the more general Zilber-Pink's conjecture on atypical intersections in Shimura varieties) for any admissible variation of mixed Hodge structures on a smooth quasi-projective variety. 
\\

This paper studies a particular case of Klingler's generalized André-Oort conjecture for pure variations of integral Hodge structures. 

\subsection{Hodge locus} Let $\VV\to S^{\text{an}}$ be a polarized \footnote{We only consider polarizable variations of Hodge structures throughout the paper.} variation of integral Hodge structure (usually abbreviated $\ZZ$-VHS) of weight $p\in\ZZ$ on a smooth irreducible complex quasi-projective variety $S$. Thus $\VV$ is a triple $(\VV_\ZZ, F^\pt, \varphi)$, where $\VV_\ZZ$ is a local system of finite free $\ZZ$-modules on the complex manifold $S^{\text{an}}$ --- the analytification of $S$, $F^\pt$ is a decreasing filtration by holomorphic subbundles on the holomorphic bundle $(\Vcal^{\text{an}} := \VV_\ZZ\otimes_{\underline{\ZZ}_{S^{\text{an}}}}\Ocal_{S^{\text{an}}}, \nabla^{\text{an}})$ satisfying Griffiths' transversality condition
\begin{equation}\label{Grifftrans}
\nabla^{\text{an}} F^\pt \subset \Omega^1_{S^{\text{an}}} \otimes_{\Ocal_{S^{\text{an}}}} F^{\pt-1}
\end{equation}
and $\varphi: \VV_\ZZ \otimes \VV_\ZZ \to \underline{\ZZ}_{S^{\text{an}}}(-p)$ is a bilinear pairing of local systems, such that $(\VV_{\ZZ, s}, F^\pt_s, \varphi_s)$ is a polarized $\ZZ$-Hodge structure of weight $p$ for all $ s\in S^{\text{an}}$. This definition is an abstraction of the geometric case corresponding to $\VV_\ZZ = (R^pf_* \ZZ_{\Xcal^{\text{an}}})_{\textnormal{prim}}/\text{torsion}$ (for $p\geq0$), the primitive part of the local system of the $p$-th integral cohomologies modulo torsion of the fibers of a smooth projective morphism $f: \Xcal \rightarrow S$ and $(\Vcal^{\text{an}}, \nabla^{\text{an}})$ the Gauss-Manin connection. Following Griffiths [cf. \cite{Sch73} Theorem (4.13)] the holomorphic bundle $\Vcal^{\text{an}}$ admits a unique algebraic structure $\Vcal$ such that the holomorphic connection $\nabla^{\text{an}}$ is the analytification of an algebraic connection $\nabla$ on $\Vcal$ which is regular, and the filtration $F^\pt\Vcal^{\text{an}}$ is the analytification of an algebraic filtration $F^\pt\Vcal$. Thus from now on we will omit $^\text{an}$ from the notations and the meaning will be clear from the context.
\\

Inspired by the rational Hodge conjecture, one would like to know how the Hodge locus $\HL(S,\VV^\otimes)\subset S$ is distributed in $S$. Here $\HL(S, \VV^\otimes)$ is by definition the subset of points $s$ of $S$ for which exceptional Hodge classes \footnote{In this paper, by a Hodge class, we will always mean a Hodge class of type $(0, 0)$.} do occur in $\VV_{\QQ, s}^m \otimes (\VV_{\QQ,s}^\vee)^n$ for some $m, n\in\ZZ_{>0}$, where $\VV_{\QQ,s}^\vee$ denotes the $\QQ$-Hodge structure dual to $\VV_{\QQ, s}$.  
\\

The Tannakian formalism available for Hodge structures gives us a particularly useful group-theoretic description of the Hodge locus $\HL(S, \VV^\otimes)$. Recall that for every $s\in S$, the \textit{Mumford-Tate group} $\Gbold_s$ of the Hodge structure $\VV_{\QQ,s}$ is the Tannakian group of the Tannakian subcategory $<\VV_{\QQ,s}^\otimes>$ of pure polarized Hodge structures tensorially generated by $\VV_{\QQ,s}$ and $\VV_{\QQ,s}^\vee$. Equivalently, the group $\Gbold_s$ is the stabilizer of the Hodge classes  in the rational Hodge structures tensorially generated by $\VV_{\QQ,s}$ and its dual.  The group $\Gbold_s$ is a connected reductive algebraic $\QQ$-group, canonically endowed with a morphism of real algebraic groups $h_s: \SS:=\Res_{\CC/\RR}\GG_m \to\Gbold_{s,\RR}$. Let $Z\subset S$ be an irreducible algebraic subvariety of $S$. A point $s$ in the smooth locus $Z^\text{sm}$ of $Z$ is said to be \textit{Hodge generic} for the restriction $\VV |_{Z^\text{sm}}$ if $\Gbold_s$ is maximal when $s$ ranges through $Z^\text{sm}$. Since $Z$ is irreducible, two Hodge generic points of $Z^\text{sm}$ have the same Mumford-Tate group, called the \textit{generic Mumford-Tate group} $\Gbold_Z$ of $(Z, \VV |_{Z^\text{sm}})$. Then the Hodge locus $\HL(S, \VV^\otimes)$ is also the subset of points of $S$ which are not Hodge generic.
\\

A fundamental result of Cattani-Deligne-Kaplan \cite{CattaniDeligneKaplan95} states that $\HL(S, \VV^\otimes)$ is a countable union of closed irreducible strict algebraic subvarieties of $S$. \\

\begin{definition} \label{special:def}
Let $\VV$ be a $\ZZ$-VHS on a smooth irreducible complex quasi-projective variety $S$. 

A closed irreducible algebraic subvariety $Z$ of $S$ is called \textit{special} for $\VV$, if it is maximal among the closed irreducible algebraic subvarieties of $S$ with the same generic Mumford-Tate group as $Z$.

Special subvarieties of dimension zero are called {\em special points} for $(S, \VV)$. 

A special point $s\in S$ whose Mumford-Tate group $\Gbold_s$ is commutative is called a \textit{CM point} for $(S, \VV)$.
\end{definition}

So by definition, if $Z\subset S$ is a special subvariety for $(S, \VV)$, then $Z$ is either contained in $\HL(S, \VV^\otimes)$ (in which case we call $Z$ \textit{strict}), or $Z = S$.\\

Choose $s\in Z^{\rm sm}$ and let $\Dcal_Z$ be the $\Gbold_Z(\RR)$-conjugacy class of $h_s$. The pair $(\Gbold_Z, \Dcal_Z)$ is called the \textit{generic Hodge datum} for $(Z, \VV |_{Z^\text{sm}})$.

\begin{definition}
	Let $Z\subset S$ be a special subvariety for $\VV$. Then $Z$ is called {\em of Shimura type} if the generic Hodge datum $(\Gbold_Z, \Dcal_Z)$ for $(Z, \VV |_{Z^\text{sm}})$ is a Shimura datum (see \cite{Mil05} Section $5$ for the definition of Shimura datum).
\end{definition}

Notice that CM points for $(S, \VV)$ are of Shimura type.\\

The problem we are interested in can be phrased vaguely as follows:

\begin{question} \label{CMpoint}
Given a $\ZZ$-VHS on a smooth irreducible complex quasi-projective variety, can we describe the distribution of its CM points, or more generally of its special subvarieties of Shimura type?
\end{question}

\subsection[Andr\'e-Oort for VHS]{Andr\'e-Oort conjecture for variations of Hodge structures}
\subsubsection{Variations of Hodge structures of Shimura type and of general Hodge type}

We keep the same notations as in the previous section. The Griffiths' transversality condition (\ref{Grifftrans}) establishes a fundamental dichotomy between $\ZZ$-VHS {\em of Shimura type} (called {\em classical} in \cite{GGK}) for which the generic Hodge datum $(\Gbold, \Dcal)$ is a Shimura datum  and $\ZZ$-VHS of general Hodge type (called {\em non-classical} in \cite{GGK}). Roughly speaking, a $\ZZ$-VHS is of Shimura type if it is an element of some $<\VV_{\QQ}^\otimes>$, for $\VV$ an effective $\ZZ$-VHS of weight $p=1$ (i.e., a family of abelian varieties), or weight $p=2$ and very restricted Hodge type (like family of $K3$-surfaces). It is the Hodge-theoretic incarnation of a family of abelian motives. On the other hand, variations of integral Hodge structures of general Hodge type form the vast majority of $\ZZ$-VHS, incarnating families of non-abelian motives.
\\

For a $\ZZ$-VHS of Shimura type, the Hodge filtration is so short that the Griffiths' transversality condition is automatically satisfied. As a result, classifying spaces do exist for $\ZZ$-VHS of Shimura type: these are exactly the Shimura varieties $\Sh_K(\Gbold, \Dcal)$ (see \cite{Mil05} Section $5$ for the definition of a Shimura datum), which are algebraic varieties (canonically defined over a number field) generalizing the moduli space $\mathcal{A}_g$ of principally polarized abelian varieties of dimension $g$. Given $\VV$ a $\ZZ$-VHS of Shimura type on $S$, there exists an algebraic classifying map $\psi: S\to \Sh_K(\Gbold, \Dcal)$ and an algebraic representation $\rho$ of $\Gbold$ such that $\VV = \psi^* \VV_\rho$. Here $\VV_\rho$ is the standard $\ZZ$-VHS on $\Sh_K(\Gbold, \Dcal)$ associated to $\rho$. Moreover the Hodge locus $\HL(S, \VV^\otimes)$ coincides with 
$\psi^{-1}(\psi(S) \cap \HL(\Sh_K(\Gbold, \Dcal))$, where the Hodge locus $\HL(\Sh_K(\Gbold,
\Dcal)):= \HL(\Sh_K(\Gbold, \Dcal), \VV_\rho^\otimes)$ is in fact independent of the choice of the faithful representation $\rho$ of $\Gbold$ and each special subvariety of $\Sh_K(\Gbold, \Dcal)$ can
be geometrically described as an irreducible component of a Hecke translate of a Shimura subvariety of $\Sh_K(\Gbold, \Dcal)$.

\subsubsection{The conjectures}  The following conjecture of Klingler proposes a characterization of the $\ZZ$-VHS with many CM points.

\begin{conjecture}[Klingler \cite{Klingler17}, Conjecture 5.3] \label{AO:conj1}
Let $\VV$ be a $\ZZ$-VHS on a smooth irreducible complex quasi-projective variety $S$ with generic Hodge datum $(\Gbold, \Dcal)$. Suppose that the set of CM points for $(S, \VV)$ is Zariski-dense in $S$. Then $(\Gbold, \Dcal)$ is a Shimura datum and we have a Cartesian diagram
\begin{equation*} \label{e2:conj1}
\begin{tikzcd}
\VV = \psi^* \VV_\rho \arrow[d] \arrow[r] & \VV_\rho \ar[d] \\
S \arrow[r, "\psi"] & \Sh_K(\Gbold, \Dcal)
\end{tikzcd}
\end{equation*}
where $\psi$ is a dominant morphism to a connected component
$\Sh_K^\circ(\Gbold, \Dcal)$ of a Shimura variety $\Sh_K(\Gbold, \Dcal)$, $\rho: \Gbold \to \GL(V)$ is an algebraic representation
and $\VV_\rho \to \Sh_K(\Gbold, \Dcal)$ is the associated standard $\ZZ$-VHS on $\Sh_K(\Gbold, \Dcal)$.
\end{conjecture}

It follows readily from these considerations that the restriction of conjecture~\ref{AO:conj1} to the class of $\ZZ$-VHS of Shimura type is equivalent to the classical Andr\'e-Oort conjecture, while the full conjecture~\ref{AO:conj1} is equivalent to both the classical Andr\'e-Oort Conjecture and the following conjecture~\ref{AO:conj2}:

\begin{conjecture} \label{AO:conj2}
Let $\VV$ be a $\ZZ$-VHS on a smooth irreducible complex quasi-projective variety $S$. Suppose that the set of CM points for $(S, \VV)$ is Zariski-dense in $S$. Then $(S, \VV)$ is of Shimura type.
\end{conjecture}

Many works have been devoted to the classical Andr\'e-Oort conjecture, culminating to its proof when a Shimura variety $\Sh_K(\Gbold, X)$ is of abelian type (see for example \cite{KUY18} for a survey). The proof of the classical Andr\'e-Oort Conjecture relies on two completely different ingredients: on the one hand a precise {\em arithmetic} analysis of the Galois orbits of CM points (lower bound and heights); on the other hand, a {\em geometric} analysis of the distribution in $\Sh_K(\Gbold, X)$ of {\em positive dimensional} special subvarieties. 
\\

In this paper we will concentrate on the {\em geometric} part of conjecture~\ref{AO:conj2}, namely on the following:

\begin{conjecture}[geometric Andr\'e-Oort for $\ZZ$-VHS, Klingler \cite{Klingler17}, Conjecture 5.7]  \label{AO:conj3}
Let $\VV$ be a $\ZZ$-VHS on a smooth irreducible complex quasi-projective variety $S$. Suppose that the set of positive dimensional special subvarieties for $(S, \VV)$, which are of Shimura type with dominant period maps, is Zariski-dense in $S$. Then $(S, \VV)$ is of Shimura type with dominant period map.
\end{conjecture}

\subsection{Statements of the main results} The main result we obtain in the direction of Conjecture \ref{AO:conj3} is the following:

\begin{theorem}  \label{NF:finite}
   	Let $\VV$ be a $\ZZ$-VHS on a smooth irreducible complex quasi-projective variety $S$. Then the union of the non-factor special subvarieties for $(S, \VV)$, which are of Shimura type with dominant period maps, is a finite union of special subvarieties of $S$.
\end{theorem}

As a corollary, we have
\begin{corollary} \label{NF:finite1}
Let $\VV$ be a $\ZZ$-VHS on a smooth irreducible complex quasi-projective variety $S$. If $S$ contains a Zariski-dense subset of non-factor special subvarieties, which are of Shimura type with dominant period maps, then $(S, \VV)$ is of Shimura type with dominant period map.
\end{corollary}

\begin{remark}
	The notion of \textit{non-factor} special subvarieties was introduced by Ullmo in \cite{Ullmo07} for Shimura varieties, as a generalization of the \textit{strongly special subvarieties} defined by Clozel and Ullmo in \cite{ClozelUllmo05}. The precise definition is given in Section \ref{NF}. The restriction to non-factor special subvarieties avoids in particular the appearance of the special points. We have no tools to deal with special points in the non-classical setting at the moment. 
\end{remark}

\begin{remark}
Recent work of Klingler and Otwinowska \cite{KO19} shows that if the adjoint group $\Gbold^{\text{ad}}$ of the generic Mumford-Tate group $\Gbold$ of $(S, \VV)$ is simple, then the union of positive special subvarieties in $\HL(S, \VV^\otimes)$ is either an algebraic subvariety of $S$ or is Zariski-dense in $S$. Here we say an irreducible algebraic subvariety Z of $S$ is \textit{positive} if the local system $\VV |_Z$ is not constant.
\end{remark}

Theorem \ref{NF:finite} is a consequence of the following equidistribution result of non-factor Shimura type special subvarieties in any connected Hodge variety, which is a generalization to the non-classical setting of Clozel and Ullmo's \cite{ClozelUllmo05} and Ullmo's \cite{Ullmo07} result on the equidistribution of positive dimensional special subvarieties in a Shimura variety.

\begin{theorem} \label{Equidis:MT}
Let $\VV$ be a $\ZZ$-VHS on a smooth irreducible complex quasi-projective variety $S$ with generic Hodge datum $(\Gbold, \Dcal)$. Let $\psi: S\to\Hod_\G(S, \VV) := \G\bs\Dcal$ be the associated period map, where $\Gamma\subset\Gbold(\QQ)$ is an arithmetic lattice. 

Let $(Z_n)$ be a sequence of non-factor special subvariety of $S$ which are of Shimura type with dominant period maps and  $(W_n)$ be the corresponding sequence of non-factor Shimura type special subvarieties in $\Hod_\G(S, \VV)$. Let $\mu_{W_n}$ be the canonical Borel probability measure on $\Hod_\G(S, \VV)$ with support $W_n$. Then there exists a special subvariety $W_\infty$ of $\Hod_\G(S, \VV)$, which is non-factor and of Shimura type, and a subsequence $(\mu_{W_{n_k}})$ of $(\mu_{W_n})$ such that $\mu_{W_{n_k}}$ is weakly convergent to $\mu_{W_\infty}$. Moreover, $W_{n_k}\subset W_\infty$ for $k\gg0$, and the irreducible component $Z_\infty$ of $\psi^{-1}(W_\infty)$ containing $Z_{n_k}, k\gg 0$ is a non-factor special subvariety of Shimura type of $S$ and $\psi$-dominant over $W_\infty$.
\end{theorem}

\subsection{Strategy of the proof} The method we use to prove Theorem \ref{Equidis:MT} is from ergodic theory, due to Ratner (\cite{Ratner91-1}, \cite{Ratner91-2}) Mozes-Shah \cite{MozesShah95} and Dani-Margulis \cite{DaniMargulis91}. And we deduce Theorem \ref{NF:finite} from Theorem \ref{Equidis:MT} and the definability of period maps. More precisely: 

\begin{enumerate}
\item[1.] Assuming the  Hodge variety $\Hod_\G(S, \VV)$ contains one non-factor Shimura type special subvariety, then these non-factor Shimura type special subvarieties will equidistribute in $\Hod_\G(S, \VV)$. We will follow the strategies used by Clozel and Ullmo in \cite{ClozelUllmo05} and Ullmo in \cite{Ullmo07}. But there are two main differences that we want to address: 

\begin{enumerate}
\item for a non-factor special subvariety $W$ of $\Hod_\G(S, \VV)$ associated to a Hodge subdatum $(\Hbold, \Dcal_\Hbold)$ of $(\Gbold, \Dcal)$ with $\Gbold$ semisimple of adjoint type, we need to show that the centralizer $\Zcent_\Gbold(\Hbold^\text{der})(\RR)$ of $\Hbold^\text{der}$ in $\Gbold$ is contained in $M_h$, the isotropy subgroup in $\Gbold(\RR)$ of a Hodge generic point $h\in\Dcal_\Hbold$. Ullmo's method doesn't apply here since the Hodge datum $(\Gbold, \Dcal)$ is in general not a Shimura datum. We give a Hodge-theoretic proof of this result (Proposition \ref{centNF}), which works in all cases.

\item we need to show the limit of a sequence of non-factor Shimura type special subvarieties is again of Shimura type. This is quite easy in the classical case as $(\Gbold, \Dcal)$ is itself of Shimura type.
\end{enumerate}

\item[2.] We need to know that there are only finitely many components in the preimage $\psi^{-1}(W)$ of a special subvariety $W$ of $\Hod_\G(S, \VV)$ under the period map. This follows from the recent result of Bakker, Klingler and Tsimerman \cite{BKT18} on the definability of period map $\psi$.
\end{enumerate}

As for the organization of this paper, in Section \ref{ergodic} we provide a recollection of the ergodic results that we need. In Section \ref{Hodvar} we recall the general definitions of a Hodge datum and of a Hodge variety and review the definability of the period map. In Section \ref{NF}, we discuss the equidistribution of non-factor special subvarieties. Section \ref{mainresult} and \ref{normalizer} give the proof of the main results. 
\\

\textbf{\textit{Notations}}. An algebraic group will be denoted by boldface (or blackboard bold) letters (e.~g. $\SS, \Gbold, \Hbold, \cdots$) and a Lie group will be denoted by usual letters (e.~g. $G, H, \cdots$).\\

Let $\Hbold$ be an algebraic group.
\begin{itemize}
\item[--] The adjoint group and derived subgroup of $\Hbold$ are denoted by $\Hbold^\text{ad}$ and $\Hbold^\text{der}$ respectively; the centralizer (resp. normalizer) of a subgroup $\Hbold$ in an algebraic group $\Gbold$ is denoted by $\Zcent_{\Gbold}(\Hbold)$ (resp. $\Nor_{\Gbold}(\Hbold)$);

\item[--] If $\Hbold$ is defined over $\RR$, we denote $\Hbold(\RR)^+$ the identity component of $\Hbold(\RR)$ for the real topology and $\Hbold(\RR)_+$ the preimage of $\Hbold^\text{ad}(\RR)^+$ under the adjoint homomorphism $\ad: \Hbold\to\Hbold^{\text{ad}}$;

\item[--] If $\Hbold$ is connected semisimple and defined over a filed $k$, then $\Hbold$ is the almost direct product of its minimal nonfinite normal $k$-subgroups $\Hbold_1,\cdots,
\Hbold_r$ (cf. \cite{Mil17} theorem 21.51). If $\Hbold$ is adjoint or simply connected, the product is direct. By abuse of language, the $\Hbold_i$ are called $k$-\textit{simple} factors of $\Hbold$.
\end{itemize}

\bigskip

\textbf{\textit{Acknowledgement}}. I wish to record my indebtedness to my supervisor Bruno Klingler during the preparation of this paper. I had numerous discussions with him, which have influenced this paper.


\section{Some ergodic results \`a la Ratner, Mozes and Shah}\label{ergodic}
In this section, we will review some results form ergodic theory (\cite{Ratner91-1}, \cite{Ratner91-2} and \cite{MozesShah95}) that will be used later. We follow the same terminologies as defined in \cite{ClozelUllmo05}, \cite{ClozelUllmo05-2} and \cite{Ullmo07}. 

\subsection{Algebraic groups of type $\Kcal$ and Lie subgroups of type $\Hcal$}

\begin{definition}
	A connected linear $\QQ$-algebraic group $\Hbold$ is said to be of type $\Kcal$ if its radical is unipotent, and $\Hbold^{\text{ss}} := \Hbold/\uRad(\Hbold)$ is of non-compact type: that is, none of its $\QQ$-simple factors are $\RR$-anisotropic. Here $\uRad(\Hbold)$ denotes the unipotent radical of $\Hbold$.  
\end{definition}

Let $\Gbold$ be a connected semisimple $\QQ$-algebraic group and $G = \Gbold(\RR)^+$ be the associated connected Lie group. Let $\G$ be an arithmetic lattice of $G$ and $\Omega$ be the homogeneous space $\G\bs G$ on which $G$ acts by right translations. Let $\Pscr(\Omega)$ denote the set of Borel probability measures on $\Omega$ equipped with the weak-$^*$ topology.

\begin{definition} (cf. \cite{ClozelUllmo05} Section 2).
Let $H$ be a connected closed Lie subgroup of $G$. Then $H$ is said to be of type $\Hscr$ if
    \begin{enumerate}
    \item[(i)] $H\cap\G$ is a lattice of $H$. In particular, the orbit $\G\backslash\G H$ of $\G e\in\Omega$ under $H$ is closed in $\Omega$ (\cite{Ratner91-1},  Proposition 1.4). We denote by $\mu_H\in\Pscr(\Omega)$ the unique $H$-invariant Borel probability measure supported on $\G\backslash\G H$ .

    \item[(ii)] The subgroup $L(H)\subset H$ generated by the one-parameter unipotent subgroups of $G$ contained in $H$ acts ergodically on $\G\backslash\G H$ with respect to the measure $\mu_H$.
    \end{enumerate}
\end{definition}

Note that the definition of an algebraic group being type $\Kcal$ is intrinsic, while the notion of type $\Hcal$ is for a subgroup of a given group. The relation between type $\Kcal$ algebraic subgroups of $\Gbold$ and type $\Hcal$ closed Lie subgroups of $G$ is given by the following lemma, proven as Lemmas 2.1, 2.2 and 2.3 in \cite{ClozelUllmo05}.

\begin{lemma} Let $\Gbold$ be a connected semisimple $\QQ$-algebraic subgroup of type $\Kcal$.\hfill \label{HK}
\begin{enumerate}
\item[(1)] If $\Hbold$ is a connected semisimple $\QQ$-subgroup of $\Gbold$ of type $\Kcal$, then $H := \Hbold(\RR)^+$ is a closed Lie subgroup of $G$ of type $\Hcal$.

\item[(2)] If $H$ is a connected Lie subgroup of $G$ of type $\Hcal$, then there exists a connected $\QQ$-algebraic subgroup $\Hbold$ of $\Gbold$ of type $\Kcal$ such that $H := \Hbold(\RR)^+$.
\end{enumerate}
\end{lemma}

\begin{remark}
The $\QQ$-algebraic subgroup $\Hbold$ in part (2) of Lemma \ref{HK} is constructed as the $\QQ$-Zariski closure of $H$ in $\Gbold$.\end{remark}

\subsection{A theorem of Mozes and Shah} Let $\Qscr(\Omega)$ be the subset of $\Pscr(\Omega)$ consisting of all the $H$-invariant Borel probability measures $\mu_H$ associated to type $\Hcal$ closed connected Lie subgroups $H$ of $G$.

\begin{theorem} [Mozes-Shah \cite{MozesShah95}, Theorem 1.1 and Corollary 1.4] \label{MS95} \hfill
\begin{enumerate}
\item[(1)]  $\Qscr(\Omega)$ is a compact subset of $\Pscr(\Omega)$.
\item[(2)]  If $(\mu_n)$ is a sequence in $\Qscr(\Omega)$ that weakly converges to $\mu\in\Qscr(\Omega)$, then the supports $\supp(\mu_n)$ are contained in $\supp(\mu)$ for $n$ big enough.
\end{enumerate} 
\end{theorem}


\section{Hodge varieties and definability of period maps}\label{Hodvar}
In this section we recall the definition of special subvarieties of a Hodge variety and give a brief review of the definability of period maps. The main references for this section are \cite{GGK}, \cite{CMP17},  \cite{Klingler17}, and \cite{BKT18}.

\subsection{Hodge data and Hodge varieties}
Let $\SS:=\Res_{\CC/\RR}\GG_{m,\CC}$ denote the Deligne torus. It is the Tannaka dual group of the category of real Hodge structures. The inclusion $\RR^\times\hookrightarrow\CC^\times$ corresponds to an inclusion of real algebraic groups $w: \GG_{m, \RR}\hookrightarrow\SS$. 

\begin{definition}\hfill

\begin{enumerate}
\item[(i)] A \textit{Hodge datum} is a pair $(\Gbold, \Dcal)$ consisting of a connected $\QQ$-reductive group and a $\Gbold(\RR)$-conjugacy class $\Dcal$ of some homomorphism $h\in\Hom(\SS,\Gbold_\RR)$ satisfying the following conditions:
\begin{enumerate}
\item[HD~$0$:] the \textit{weight homomorphism} $w_h := h\circ w: \GG_{m, \RR}\to\Gbold_\RR$ is a cocharacter of the center of $\Gbold_\RR$ and is defined over $\QQ$;
\item[HD~$1$:] the involution $\Int(h(\sqrt{-1}))$ is a Cartan involution of the adjoint group $\Gbold^{ad}_\RR$.
\end{enumerate}
\item[(ii)] Let $(\Gbold, \Dcal)$ be a Hodge datum and $\Dcal^+$ be a connected component of $\Dcal$. The pair $(\Gbold, \Dcal^+)$ is then called a \textit{connected Hodge datum}.
\item[(iii)] A (connected) Hodge datum $(\Gbold, \Dcal)$ (resp. $(\Gbold,\Dcal^+)$) is said to be of \textit{Shimura type} if it satisfies two more conditions:
\begin{enumerate}
\item[HD~$2$:] the Hodge structure induced on the Lie algebra $\Lie(\Gbold_\RR)$ by $\Ad\circ h$ is of type
$$\{(-1, 1), (0, 0), (1, -1)\}.$$
\end{enumerate}
\begin{enumerate}
\item[HD~$3$:] $\Gbold^{\text{ad}}$ has no $\QQ$-factor on which the projection of $h$ is trivial. By the presence of axioms HD $1$ and HD $2$, this is equivalent to say that $\Gbold^{\text{ad}}$ is of non-compact type.
\end{enumerate}
\end{enumerate}
\end{definition}

\begin{remark}\hfill
\begin{itemize}
	\item[(1)] For $(\Gbold, \Dcal)$ a Hodge datum, there exists a unique structure of complex manifold on $\Dcal$ such that for some (any) faithful (finite dimensional, algebraic) representation of $\Gbold$, the associated family of Hodge structures on $\Dcal$ varies holomorphically (cf. \cite{Mil05} Theorem 2.14).
	\item[(2)] Any discrete subgroup $\G$ of $\Gbold(\QQ)_+ := \Gbold(\QQ)\cap\Gbold(\RR)_+$\footnote{Recall that $\Gbold(\RR)_+$ is the stabilizer of $\Dcal^+$ in $\Gbold(\RR).$} acts properly discontinuously on $\Dcal^+$, so that $\G\bs\Dcal^+$ is a complex
          analytic space with at most finite quotient singularities (cf. \cite{CMP17} Section 16.3).

\end{itemize}
	
\end{remark}

\begin{definition}\hfill
\begin{enumerate}
\item[(i)] Let $(\Gbold, \Dcal)$ be a Hodge datum and $K$ be a compact open subgroup of $\Gbold(\AA_f)$ where $\AA_f$ is the ring of finite adèles of $\QQ$. The \textit{Hodge variety} is defined as
$$\Hod_K(\Gbold, \Dcal) := \Gbold(\QQ)\bs \Dcal\times\Gbold(\AA_f)/K,$$
where $\Gbold(\QQ)$ acts diagonally on $\Dcal$ and $\Gbold(\AA_f)$ on the left and $K$ acts on $\Gbold(\AA_f)$ on the right.
\item[(ii)] Let $(\Gbold, \Dcal^+)$ be a connected Hodge datum. A \textit{connected Hodge variety} associated to $(\Gbold, \Dcal^+)$ is defined as the quotient $\G\bs\Dcal^+$ for an arithmetic subgroup $\G$ of $\Gbold(\QQ)_+$.
\end{enumerate}
\end{definition}

\begin{remark}\hfill
\begin{enumerate}
\item[(1)] As in the case of Shimura varieties, every connected Hodge variety is a connected component of a Hodge variety and vice versa (cf. \cite{CMP17} Lemma 16.3.8). If $K$ (resp. $\G$) is chosen sufficiently small, then $\Hod_K(\Gbold, \Dcal)$ (resp. $\G\bs\Dcal^+$) is a complex manifold and the map $\Dcal\to\Hod_K(\Gbold, \Dcal)$ (resp. $\Dcal^+\to\G\bs\Dcal^+)$ is unramified.
\item[(2)] In general, the Hodge variety $\Hod_K(\Gbold, \Dcal)$ (resp. connected Hodge variety $\G\bs\Dcal^+$) does not admit any algebraic structure (see \cite{GRT14} Theorem 1.4).
\end{enumerate}
\end{remark}

We will only consider connected Hodge data and connected Hodge varieties in this paper.

\begin{definition}
A \textit{Hodge morphism} of connected Hodge data $(\Gbold, \Dcal^+)\to(\Gbold^\prime, \Dcal^{\prime +})$ is a homomorphism of $\QQ$-algebraic groups $\varphi: \Gbold\to\Gbold^\prime$ which induces a map $\Dcal^+\to\Dcal^{\prime +}, h\mapsto\ \varphi\circ h$. A \textit{Hodge morphism} of connected Hodge varieties is a morphism of varieties induced by a morphism of connected Hodge data.
\end{definition}

\begin{remark}\label{adjoint1}
Let $h\in\Dcal^+$ and let $\Dcal^{\text{ad}, +}$ be the $\Gbold^{\text{ad}}(\RR)^+$-conjugacy class of the composition $h^\text{ad}: \SS\to\Gbold_\RR\to\Gbold^\text{ad}_\RR$. Then $\Dcal^+\cong\Dcal^{\text{ad}, +}$ and we have a morphism of connected Hodge data $(\Gbold, \Dcal^+)\to(\Gbold^{\text{ad}}, \Dcal^{\text{ad}, +})$.
\end{remark}

\subsection{Special subvarieties of a connected Hodge variety} Let $(\Gbold, \Dcal^+)$ be a connected Hodge datum and let $Y$ be a connected Hodge variety associated to $(\Gbold, \Dcal^+)$.
\begin{definition}\label{Special1}
The image of any Hodge morphism $W\to Y$ between connected Hodge varieties is called a \textit{special subvariety} of $Y$. It is said to be of \textit{Shimura type} if the connected Hodge datum corresponds to $W$ is a Shimura datum.
\end{definition}

For any special subvariety of $Y$, the Hodge morphism in Definition \ref{Special1} can be chosen such that the underlying homomorphism of algebraic groups is injective. Hence any special subvariety of $Y$ can be regarded as given by a Hodge subdatum.

\subsection{Special subvarieties associated to a $\ZZ$-VHS}

Let $\VV$ be a $\ZZ$-VHS on a smooth irreducible complex quasi-projective variety $S$. Let $\Gbold$ be its generic Mumford-Tate group. Fix a Hodge generic point $o\in S$. The Hodge structure on the fiber $\VV_{\QQ, o}\cong V$ \footnote{The pullback of the local system $\VV_\QQ$ to the topological universal cover $\hat{S}$ of $S$ is constant, hence isomorphic to $\hat{S}\times V$ for some finite dimensional $\QQ$-vector space $V$.} induces a morphism of $\RR$-algebraic groups $h_o: \SS\to\Gbold_\RR$. Let $\Dcal$ be the $\Gbold(\RR)$-conjugacy class of $h_o$ and let $\Dcal^+$ be a connected component of $\Dcal$ containing $h_o$. Then we get a connected Hodge datum $(\Gbold, \Dcal^+)$.\\

Let $\G$ be a neat arithmetic lattice of $\Gbold(\RR)_+$, the stabilizer of $\Dcal^+$ in $\Gbold(\RR)$. After passing to a finite \'etale covering of $S$, we may assume that $\G$ contains the monodromy group, namely the image of $\pi_1(S)$ in $\GL(\VV_{\ZZ, o})$. We denote by $\Hod^\circ_\G(S, \VV)$ the connected Hodge variety $\G\bs\Dcal^+$ associated to $(S, \VV)$. This is an arithmetic quotient $\G\bs\Gbold(\RR)_+/M_o$ in the sense of \cite{BKT18}. Here $M_o$ is the intersection of the isotropy subgroup of $h_o$ in $\Gbold(\RR)$ with $\Gbold(\RR)_+$, whose image in $\Gbold^{ad}(\RR)^+$ turns out to be compact. And we have the period map:
$$\psi: S\to\Hod^\circ_\G(S, \VV),$$
which is holomorphic, locally liftable and all the local liftings are horizontal.\\

Let $K = \Zcent_{\Gbold(\RR)}(h_o(\sqrt{-1}))\cap\Gbold(\RR)_+$.  Then $M_o\subset K$. And we have a canonical projection
$$\omega: \Dcal^+ = \Gbold(\RR)_+/M_o\longrightarrow\Gbold(\RR)_+/K.$$
Let $\gfrak := \Lie(\Gbold_\RR), \kfrak := \Lie(K)$, and $\mfrak := \Lie(M_o)$. Then $\gfrak$ carries a weight $0$ Hodge structure
$$\gfrak_\CC = \bigoplus\gfrak^{-j, j}$$
polarized by minus the Killing form of $\gfrak$. And by the axiom (HD 1), we have a Cartan decomposition 
$$\gfrak = \kfrak\oplus\pfrak.$$
Let $T_\omega$ (resp. $T_{\omega}^{\perp}$) be the subbundle of the tangent bundle $T_{\Dcal^+}$ of $\Dcal^+$ associated to the adjoint representation of $M_o$ on $\kfrak/\mfrak$ (resp. $\pfrak$). We then have a canonical splitting
$$T_{\Dcal^+} = T_\omega\oplus T_{\omega}^{\perp}.$$
The subbundle $T_\omega$ is holomorphic as the fibers of $\omega$ are complex submanifolds of $\Dcal^+$, while $T_{\omega}^{\perp}$ in general admits no complex structure. However, there is a holomorphic subbundle $T_{\Dcal^+}^h$contained in the complexification $T_{\omega}^{\perp}\otimes\CC$, namely the subbundle associated to the adjoint representation of $M_o$ on $g^{-1, 1}$ and we call it the holomorphic horizontal tangent bundle. When we say the period map $\psi$ is horizontal, we mean
$$d\hat{\psi}(T_{\hat{S}})\subset \hat{\psi}^*T_{\Dcal^+}^h,$$
where $\hat{S}$ is the topological universal cover of $S$ and $\hat{\psi}$ is the lifting of $\psi$ to $\hat{S}$. 

Given an irreducible algebraic subvariety $Z$ of $S$, let $\Ztilde\to Z$ be its normalization and $\Ztilde^{\text{sm}}$ be the smooth locus of $\Ztilde$. Let $u: \Ztilde^{\text{sm}}\hookrightarrow\Ztilde\to Z\hookrightarrow S$ be the composition.  Then the local system $u^*\VV$ on $\Ztilde^{\text{sm}}$ is a $\ZZ$-VHS, and we denote its generic Mumford-Tate group by $\Gbold_Z$. Let 
$$\psitilde_Z: \Ztilde^{\text{sm}}\to\Hod^\circ_{\G_Z}(\Ztilde^{\text{sm}}, u^*\VV) = \G_Z\bs\Dcal_Z^+$$ 
be the associated period map, where $\G_Z = \G\cap\Gbold_Z(\QQ)$, we then have a commutative diagram
\begin{equation*} \label{Specialclosure}
\begin{tikzcd}
\Ztilde^{\text{sm}} \arrow[d, "u"] \arrow[r, "\psitilde_Z"] & \Hod^\circ_{\G_Z}(\Ztilde^{\text{sm}}, u^*\VV) \arrow[d, "\iota_Z"] \\
S \arrow[r, "\psi"] & \Hod^\circ_\G(S, \VV).
\end{tikzcd}
\end{equation*}

Notice that the restriction of the period map $\psi$ to the smooth locus of $Z$ factors through the special subvariety $\im(\iota_Z)$ of $\Hod^\circ_\G(S, \VV)$ and every complex analytic irreducible component of the preimage $\psi^{-1}(\im(\iota_Z))$ is a special subvariety of $S$ for $\VV$. Conversely, if $Z$ is a special subvariety for $(S, \VV)$, then it follows readily from the Definition \ref{special:def} that $Z$ is a  
complex analytic irreducible component of the preimage $\psi^{-1}(\im(\iota_Z))$. We thus prove the following lemma (the last assertion is obvious):

\begin{lemma} The special subvarieties for $(S,\VV)$ are precisely the preimages of the special subvarieties for $\Hod^\circ_\G(S, \VV)$. Moreover, the preimages of special points \footnote{Zero-dimensional special subvarieties (namely, special points) for a Hodge variety are precisely the CM points (cf. \cite{CMP17} examples 16.3.7)} in $\Hod^\circ_\G(S, \VV)$ are CM points for $(S, \VV)$.
\end{lemma}

\begin{remark} 
It can be shown that the set of CM points is dense in $\Hod^\circ_\G(S, \VV)$ (cf. \cite{CMP17} Corollary 17.1.5). However, there is no guarantee for the image of the period map $\psi(S)$ to contain even one CM point. 
\end{remark}

\begin{definition}
	Let $Z$ be an irreducible algebraic subvariety of $S$. The \textit{algebraic monodromy group} $\Hbold_Z$ of $Z$ for $\VV$ is defined to be the Zariski closure in $\GL(V)$ of the monodromy group of the local system $u^*\VV_\ZZ$ on $\Ztilde^{\text{sm}}$.\end{definition}
 
\subsection{Definability of period maps and algebraicity of special subvarieties} Although the period map $\psi$ is transcendental, Bakker, Klingler and Tsimerman \cite{BKT18} showed that it has moderate geometry in the sense of tame topology. For a reference to the notions of tame topology and definability in some o-minimal structure (for instance $\RR_{\rm alg}, \RR_{\rm an}, \RR_{\rm an, exp}, \cdots$), see \cite{vdD98}. 

\begin{theorem}[Bakker, Klingler and Tsimerman]\label{BKT}\hfill
\begin{enumerate}
\item[(1)] There is a natural $\RR_{\rm alg}$-definable manifold structure on the connected Hodge variety $\Hod^\circ_\G(S, \VV)$.
\item[(2)] With respect to the $\RR_{\rm an, exp}$-definable manifold structure extending the $\RR_{\rm alg}$-definable manifold structure on $S$ (resp. on $\Hod^\circ_\G(S, \VV)$) coming from its complex algebraic structure (resp. defined in part (1)), the period map $\psi: S\to\Hod^\circ_\G(S, \VV)$ is $\RR_{\rm an, exp}$-definable.
\item[(3)] For any special subvariety $Y$ of $\Hod^\circ_\G(S, \VV)$, the preimage $\psi^{-1}(Y)$ is an algebraic subvariety of $S$. In particular, $\psi^{-1}(Y)$ has only finitely many irreducible components.
\end{enumerate}
\end{theorem}

\begin{proof}
Part (1) is Theorem 1.1 (1) in \cite{BKT18}, part (2) is Theorem 1.3 in \cite{BKT18} and part (3) is Theorem 1.6 in \cite{BKT18} .
\end{proof}

\subsection{The structure theorem for period maps} For later use, we need a structure theorem for period maps.
\\

Let $\Hbold$ be the algebraic monodromy group of $S$ for $\VV$. It follows from \cite{An92} Theorem 5.1 that $\Hbold$ is a normal subgroup of derived subgroup $\Gbold^{\text{der}}$ of the generic Mumford-Tate group $\Gbold$. As $\Gbold^{\text{der}}$ is semisimple, there exists a normal subgroup $\Fbold$ of $\Gbold^{\text{der}}$ such that $\Gbold^{\text{der}}$ is an almost direct product of $\Hbold$ and $\Fbold$. Let $\Hbold^{\text{nc}}$ and $\Hbold^{\text{c}}$ be the non-compact and compact part of $\Hbold$ respectively. Then we will have an isogeny of $\QQ$-reductive groups:
$$\Hbold^{\text{nc}} \times \Hbold^{\text{c}} \times \Fbold\to \Gbold^{\text{der}}$$
which induces a surjective holomorphic map with finite fibers between Mumford-Tate domains:
$$\Dcal^+_{\Hbold^{\text{nc}}} \times \Dcal^+_{\Hbold^{\text{c}}} \times \Dcal^+_\Fbold\to \Dcal^+.$$
Here for a $\QQ$-algebraic subgroup $\Gbold^\prime$ of $\Gbold$, we write $\Dcal_{\Gbold^\prime}$ for the $\Gbold^\prime(\RR)$-orbit in $\Dcal$ of a fixed lifting in $\Dcal$ of the image $\psi(o)\in\Hod^\circ_\G(S, \VV)$ and $\Dcal^+_{\Gbold^\prime}$ the connected component of $\Dcal_{\Gbold^\prime}$ containing the lifting.

\begin{theorem} \label{StrPer}
Let $\VV$ be a $\ZZ$-VHS on a smooth irreducible complex quasi-projective variety $S$. Then its associated period map $\psi: S\to\Hod^\circ_\G(S, \VV)$ factors as:
\begin{equation*}
\psi = (\psi_{\text{nc}}, \psi_{\text{c}}, \psi_f): S\longrightarrow \G^{\text{nc}}\bs\Dcal^+_{\Hbold^{\text{nc}}} \times \G^{\text{c}}\bs \Dcal^+_{\Hbold^{\text{c}}} \times \Dcal^+_\Fbold,
\end{equation*}
where $\G^{\text{nc}} := \G\cap\Hbold^{\text{nc}}(\QQ)$ and $\G^{\text{c}} := \G\cap\Hbold^{\text{c}}(\QQ)$. Moreover,
\begin{enumerate}
\item[(1)] \label{trivial} the component $\psi_f$ is constant; correspondingly, the $\ZZ$-VHS $\VV$ is a direct sum of a sub-VHS whose generic Mumford-Tate group is the whole group $\Gbold$ and a(n) (iso-)trivial one. Let $e_3$ be image of $S$ under $\psi_f$.
\item[(2)] for any point $x\in\Dcal^+_{\Hbold^{\text{c}}}\subset\Dcal^+$, we have $T_{\Dcal^+_{\Hbold^{\text{c}}},x}\subset T_{\omega,x}$. As a consequence, for any $e_1\in \psi_{\text{nc}}(S)$, the image $\psi(S)$ intersects $e_1\times \G^{\text{c}}\bs \Dcal^+_{\Hbold^{\text{c}}} \times e_3$ in finitely many point.
\item[(3)] the number of points of the intersections of $\psi(S)$ with $e_1\times \G^{\text{c}}\bs \Dcal^+_{\Hbold^{\text{c}}} \times e_3$ is uniformly bounded as $e_1$ varies in $\psi_{\text{nc}}(S)$.
\end{enumerate}
\end{theorem}

\begin{proof}
	For the proof of $(1)$ and $(2)$, see Chapter 15 of \cite{CMP17}. Let us prove $(3)$. Since $\G^{\text{c}}\bs \Dcal^+_{\Hbold^{\text{c}}}$ is compact, the projection 
	$$\G^{\text{nc}}\bs\Dcal^+_{\Hbold^{\text{nc}}} \times \G^{\text{c}}\bs \Dcal^+_{\Hbold^{\text{c}}}\longrightarrow\G^{\text{nc}}\bs\Dcal^+_{\Hbold^{\text{nc}}}$$
	is $\RR_{\rm an}$-definable. Since the period map $\psi$ is $\RR_{\rm an, exp}$-definable by Theorem \ref{BKT}, the map 
	$$\psi(S)\longrightarrow\psi_{nc}(S)$$
	is $\RR_{\rm an, exp}$-definable. Since each fiber of this map is finite by $(2)$, the uniformly boundedness then follows from the finiteness lemma (the Lemma $1.7$ in Chapter 3, Section 1 of \cite{vdD98}). 
\end{proof}


\section{non-factor special subvarieties} \label{NF}

\subsection{} \label{NF_1} In this section, we introduce the notion of non-factor special subvarieties (Definition \ref{NFdef}). This is a natural definition from the equidistribution point of view: as explained in \cite{Ullmo07}, for a sequence $(Y_n)$ of special subvarieties of $\Hod^\circ_\G(S, \VV)$, we cannot expect in general that the associated sequence of Borel probability measures $\mu_n = \mu_{Y_n}$ on $\Hod^\circ_\G(S, \VV)$ with support $Y_n$  weakly converges. For example, if $(Y_n)$ is a sequence of special points in $\Hod^\circ_\G(S, \VV)$,  then $\mu_n$ is just the Dirac measure supported at the point $Y_n$. Such a sequence can converge to a non special point or may tend to $\infty$. Even for positive dimensional special subvarieties the same problem may occur. Start with a special subvariety of $\Hod^\circ_\G(S, \VV)$ of the form $Y\times Y^\prime$ for two special subvarieties $Y$ and $Y^\prime$. Let $(y_n)$ be a sequence of special point of $Y^\prime$ and $Y_n = Y \times \{y_n\}$, then there is no hope of proving the weak convergence of $\mu_n$.

\begin{definition} [cf. \cite{Ullmo07}] \label{NFdef} \hfill
\begin{enumerate}
\item[(i)] Let $Y$ be a connected Hodge variety. A special subvariety $W$ of $Y$ is called  \textit{non-factor} if there exists no finite morphism of connected Hodge varieties: $$W_1\times W_2\to Y$$
with $W_2$ having positive dimension, such that $W$ is the image of $W_1\times\{x\}$ in $Y$ for any (necessary special) point $x$ of $W_2$. 

\item[(ii)] Let $\VV$ be a $\ZZ$-VHS on a smooth irreducible complex quasi-projective variety $S$ and let $\Hod_\G^\circ(S, \VV)$ be the associated connected Hodge variety. A special subvariety $Z$ for $(S, \VV)$ is called \textit{non-factor} if $\G_Z\bs\Dcal_Z^+$ is a non-factor special subvariety of $\Hod_\G^\circ(S, \VV)$.
\end{enumerate}
\end{definition}

\begin{remark} Note that any Hodge variety $Y$ itself is non-factor. Assume that $Y$ is of positive dimensional. For a special point $x\in Y$, the projection
$$\{x\}\times Y\to Y$$
is a finite morphism. This shows that special points are not non-factor special subvarieties of connected Hodge variety. 
\end{remark}

\begin{remark}\label{adjoint2}
	A special subvariety which contains a non-factor special subvariety is automatically non-factor. And $W$ is a non-factor special subvariety of $Y$ if and only if $W^\text{ad}$ is a non-factor special subvariety of $Y^\text{ad}$.
\end{remark}

There is a useful group-theoretic characterization of non-factor special subvarieties for a variation of Hodge structure.

Let $\VV$ be a $\ZZ$-VHS on a smooth irreducible complex quasi-projective variety $S$ with associated Hodge datum $(\Gbold, \Dcal^+)$. We assume that $\Gbold$ is semisimple of adjoint type. Let $Z$ be a special subvariety for $(S, \VV)$ with associated Hodge subdatum $(\Gbold_Z, \Dcal_Z^+)\hookrightarrow(\Gbold, \Dcal^+)$. Let $h$ be a Hodge generic point in $\Dcal_Z^+$ and denote its isotropy group in $\Gbold(\RR)$ by $M_h$. Note that $M_h$ is a compact subgroup of $\Gbold(\RR)$.

\begin{proposition}\label{centNF} 
If $Z$ is a non-factor special subvariety for $(S, \VV)$,
then the centralizer $\Zcent_\Gbold(\Gbold_Z^\text{der})(\RR)$ is contained in $M_h$.
\end{proposition}
\begin{proof}
Let $\mathfrak{g}:=\Lie(\Gbold), \mathfrak{h}:=\Lie(\Gbold_Z^{\text{der}})$ and $\mathfrak{c}:=\Lie(\Zcent_\Gbold(\Gbold_Z^\text{der}))$. Since $\Gbold_Z^{\text{der}}$ is semisimple, the Lie algebra $\mathfrak{g}$ decomposes as a $\Gbold_Z^{\text{der}}$-module as follows:
\begin{equation} \label{decomp1}
\mathfrak{g} = \mathfrak{h} \oplus \mathfrak{c} \oplus \mathfrak{l}.
\end{equation}
where $\mathfrak{l}$ is the orthogonal complement of $\mathfrak{h} \oplus \mathfrak{c}$ with respect to the Killing form of $\gfrak$.

Notice that $\mathfrak{g}$ carries a natural weight $0$ polarized rational Hodge structure and this defines a variation of Hodge structure $\VV_{\mathfrak{g}}$ on $S$. Let $\Hbold_Z$ be the algebraic monodromy group of $Z$ for $\VV$. It follows from \cite{An92} Theorem 5.1 that $\Hbold_Z$ is a normal subgroup of $\Gbold_Z^{\text{der}}$. So the above decomposition (\ref{decomp1}) of $\mathfrak{g}$ induces a decomposition of the underlying local system $\VV_{\mathfrak{g}}$:
\begin{equation} \label{decomp2}
\VV_{\mathfrak{g}} = \VV_{\mathfrak{h}} \oplus \VV_{\mathfrak{c}} \oplus \VV_{\mathfrak{l}}.
\end{equation}

Let $\Gbold^{1}$ be the connected $\QQ$-algebraic subgroup of $\Gbold$ with Lie algebra $\mathfrak{h} \oplus \mathfrak{c}$. Then $\Gbold^1$ is reductive. And it can be seen easily that $\Gbold^{1}$ is the connected $\QQ$-subgroup of $\Gbold$ generated by $\Gbold_Z$ and $(\Zcent_\Gbold(\Gbold_Z^\text{der}))^{\circ}$. In fact, $\Gbold^1$ is the identity component of the normalizer of $\Gbold_Z^\text{der}$ in $\Gbold$. 

Let $\Dcal^{1} \subset \Dcal$ be the $\Gbold^{1}(\RR)$-orbit of $h\in\Dcal^+_Z$. Then $(\Gbold^1, \Dcal^1)$ is a Hodge subdatum. By (\ref{decomp2}), we have a decomposition of the connected Hodge subdatum $(\Gbold^1, \Dcal^{1, +})$ and a finite morphism
\begin{equation*}
(\Gbold^1, \Dcal^{1, +}) \cong (\Gbold_Z, \Dcal_Z^+) \times (\Zcent_\Gbold(\Gbold_Z^\text{der}))^{\circ}, \Dcal^{2, +})\to (\Gbold, \Dcal).
\end{equation*}

If $\Zcent_\Gbold(\Gbold_Z^\text{der})(\RR)$ is not contained in $M_h$, then $\Dcal^{2, +}$ is of positive dimensional and $\G_Z\bs\Dcal_Z^+$ is the image of a $\G_Z\bs\Dcal_Z^+\times\{x_2\}$ in $\Hod^\circ_\G(S, \VV)$. So $Z$ is not non-factor, which is a contradiction.
\end{proof}

\subsection{A description of special subvarieties}\label{description} Note that any point $h\in\Dcal^{+}$ induces a projection map
$$
\begin{aligned}
&\pi_h: & \Omega:=\G\backslash\Gbold(\RR)^+ &\rightarrow& \G\bs\Dcal^+ &= \Hod_\G^\circ(S, \VV)&\\
 & &  [g]&\mapsto&[gh].&&
\end{aligned}
$$

Let $Z$ be a special subvariety of $S$. We denote by $(\Gbold_Z, \Dcal_Z^+)$ the corresponding connected Hodge subdatum and $W$ the corresponding special subvariety of $\Hod_\G^\circ(S, \VV)$. For $h\in\Dcal_Z^+$, let $M_h :=\Zcent_{\Gbold(\RR)}(h)$ be the stabilizer of $h$ in $\Gbold(\RR)$. Then $M_h$ is a compact subgroup containing the center of $\Gbold_Z(\RR)_+$. Hence we have the following description of $W$:
\begin{eqnarray*}
W &=& \G \backslash \G\Gbold_Z(\RR)_+ h \\
  &=& \G \backslash \G\Gbold_Z^{\text{der}}(\RR)^+ h=\pi_h(\G \backslash \G\Gbold_Z^{\text{der}}(\RR)^+)\\
  &\cong &\G\backslash\G\Gbold_Z^{\text{der}}(\RR)^+M_h/M_h.
\end{eqnarray*}

Let $\mu_Z$ be the unique $\Gbold_Z^{\text{der}}(\RR)^+$-invariant Borel probability measure supported on $\G \backslash \G\Gbold_Z^{\text{der}}(\RR)^+$ and $\mu_W := (\pi_h)_*\mu_Z$. As there is a canonical $\Gbold_Z^{\text{der}}(\RR)^+$-invariant metric on $\Dcal_Z^+$, the measure $\mu_W$ is the same as the the normalized measure induced from the Hermitian metric. In particular the probability measure $\mu_M$ is independent of the choice of $h\in\Dcal_Z^+$.\\

Let $\g\in\G, \Gbold_{Z,\g}=\g\Gbold_Z\g^{-1}, h_\g=\g\cdot h$ and $\Dcal_{Z,\g}$ the $\Gbold_{Z,\g}(\RR)$-conjugacy class of $h_\g$. We also have 
$$W=\pi_{h_\g}(\G\backslash\G\Gbold^{\text{der}}_{Z,\g}(\RR)^+).$$
Fixing a fundamental domain $\Fcal$ for the action of $\G$ on $\Dcal^+$, we can thus choose $h\in\Fcal$ in the description of $W$.

\subsection{Non-factor special subvarieties and recurrence to compact sets}

We keep the same notations as in the sections \ref{NF_1} and \ref{description}. The following theorem is a corollary of a deep result of Dani and Margulis on the quantitative recurrence to compact sets for unipotent flows on $\Omega = \G\bs\Gbold(\RR)^+$ (\cite{DaniMargulis91} Theorem 2). It tells us that unipotent flows never send lattices off to infinity, which (in principle) allows us to argue “as if” $\Omega$ was compact when considering unipotent flows. This will be a key ingredient in our proof of the main theorem. It was used by Clozel and Ullmo (cf. \cite{ClozelUllmo05} Lemma 4.5) and Ullmo \cite{Ullmo07} in their proof of equidistribution of strongly (more generally, non-factor) special subvarieties in a Shimura variety and it is not difficult to adapt their arguments to our situation.
\begin{theorem}\label{DM}
There exists a compact subset $C$ of $\Hod_\G^\circ(S, \VV)$ such that $\G_Z\bs\Dcal_Z^+\cap C\neq\varnothing$ for any non-factor special subvariety $Z$ of Shimura type for $(S, \VV)$.
\end{theorem}
\begin{proof}
	It follows easily from \cite{DaniMargulis91} Theorem 2, that there exists a compact subset $C^\prime$ of $\Omega$ such that for all unipotent one-parameter subgroup $U\subset\Gbold(\RR)^+$ and $g\in\Gbold(\RR)^+$, if 
	$$\G\bs\G gU\cap C^\prime = \varnothing,$$
	then there exist a proper $\QQ$-parabolic subgroup $\Pbold^\prime$ of $\Gbold$ such that
	$$gUg^{-1}\subset\Pbold^\prime(\RR).$$
	Let $V\subset\Gbold(\RR)^+$ be a compact neighborhood of the identity element $e\in\Gbold(\RR)^+$. Then
	$$C^{\prime\prime} := C^\prime V = \{cv|c\in C, v\in V\}$$
	is also a compact subset of $\Omega$. Fix a point $h_0\in\Dcal^+$ and let $C = \pi_{h_0}(C^{\prime\prime})$. Then for any point $\alpha\in V$, we have
	$$\pi_{h_\alpha}(C^\prime)\subset C,$$
	where $h_\alpha = \alpha\cdot h_0$.
	
	For $h\in\Dcal_Z^+$, since $\Gbold(\QQ)^+$ is dense in $\Gbold(\RR)^+$, there exists $\alpha\in V$ and $\g\in\Gbold(\QQ)^+$ such that $h = \g\alpha\cdot h_0$. We then have
	\begin{eqnarray*}
\G_Z\bs\Dcal_Z^+ &=& \G \bs \G\Gbold_Z^{\text{der}}(\RR)_+\g\alpha\cdot h_0 \\
                 &=& \G \bs \G\g\g^{-1}\Gbold_Z^{\text{der}}(\RR)^+\g\alpha\cdot h_0\\
                 &=& \pi_{h_\alpha}(\G \bs \G\g\Gbold_{Z,\g}^{\text{der}}(\RR)^+)
\end{eqnarray*}
where $\Gbold_{Z,\g} := \g^{-1}\Gbold_Z\g$. If $\G_Z\bs\Dcal_Z^+\cap C = \varnothing$, then a fortiori $\G_Z\bs\Dcal_Z^+\cap \pi_{h_\alpha}(C^\prime) = \varnothing$ and hence
$$\G \bs \G\g\Gbold_{Z,\g}^{\text{der}}(\RR)^+\cap C^\prime = \varnothing.$$
Since $Z$ is of Shimura type, $\Gbold_{Z,\g}^{\text{der}}$ is of type $\Kcal$, and hence $\Gbold_{Z,\g}^{\text{der}}(\RR)^+$ is of type $\Hcal$. Then by \cite{ClozelUllmo05} Lemma 4.4, there exist a proper $\QQ$-parabolic subgroup $\Pbold$ of $\Gbold$ such that 
$$\Gbold_{Z,\g}^{\text{der}}\subset\Pbold.$$
But by Proposition \ref{centNF} and \cite{EMS97} Lemma 5.1, this cannot happen if $Z$ is non-factor.
\end{proof}


\section{Proof of the main results} \label{mainresult}

We now have all the necessary ingredients for the proof of the Theorems \ref{Equidis:MT} and \ref{NF:finite}.

Let $\VV$ be a $\ZZ$-VHS on a smooth irreducible complex quasi-projective variety $S$ and $(\Gbold,\Dcal^+)$ be the associated connected Hodge datum. By part (\ref{trivial}) of Theorem \ref{StrPer}, we may and will assume that $\VV$ has no isotrivial factors. 
\subsection{Proof of Theorem \ref{Equidis:MT}} \hfill 

We remark first that we can reduce to the case where $\Gbold$ is of adjoint type: this results from Remarks \ref{adjoint1} and \ref{adjoint2}, and the evident compatibility between the canonical measures associated to non-factor special subvarieties of $\G\bs\Dcal^+$ and of $\G^{\ad}\bs\Dcal^{\ad, +}$. By the structure theorem of period maps (Theorem \ref{StrPer}) we will assume that $\Gbold$ is adjoint of non-compact type. Let us fix a fundamental domain $\Fcal$ of $\Dcal^+$ for the action of $\G$.\\

\textbf{Step 1. Construction of the the limit.} \\

Let $(Z_n)_{n\in\NN}$ be a sequence of non-factor special subvarieties of $S$, which are of Shimura type with dominant period maps. We denote by $(\Gbold_n, \Dcal_n^+)_{n\in\NN}$ the corresponding sequence of Hodge subdata of $(\Gbold, \Dcal^+)$ and $(W_n)_{n\in\NN}$ the corresponding sequence of non-factor special subvariety of $\Hod^\circ_\G(S, \VV)$. 

For each $n\in\NN$, by the description in Section \ref{description} of special subvarieties for a variation of Hodge structure, we can write $W_n$ as
$$W_n=\pi_{h_n}(\G\backslash\G\Gbold^{\text{der}}_n(\RR)^+)$$ 
for any $h_n\in\Dcal_n^+\cap\Fcal$. By Theorem \ref{DM}, there exists a compact subset $C$ of $\Fcal$ such that $C\cap \Dcal_n^+\neq\varnothing$. We can thus choose $h_n\in C\subset\Fcal$. Since by assumption $(\Gbold_n, \Dcal_n^+)$ is a connected Shimura datum, the $\Gbold_n^{\text{der}}$ is of type $\Kscr$ and hence $\Gbold_n^{\text{der}}(\RR)^+$ is a type $\Hscr$ connected closed Lie subgroups of $\Gbold(\RR)^+$ by Lemma \ref{HK}. Let $(\mu_n)_{n\in\NN}$ be the sequence in $\Pscr(\Omega)$ of the canonical Borel probability measures supported on $\G\backslash\G\Gbold_n^{\text{der}}(\RR)^+$. By Theorem \ref{MS95}, there exists a connected Lie subgroup $F$ of $\Gbold(\RR)^+$ of type $\Hscr$ such that after possibly passing to a subsequence 

\begin{enumerate}\label{property}
\item[(a)] $(\mu_n)_{n\in\NN}$ weakly converges to $\mu_F$;
\item[(b)] $\supp (\mu_n)=\G\bs\G\Gbold_n^{\text{der}}(\RR)^+\subset\G\bs\G F$, for $n\gg 0$;
\item[(c)] the sequence $h_n$ converges to $h\in C\subset\Fcal$.
\end{enumerate}

Let $\Hbold$ be the smallest $\QQ$-algebraic subgroup of $\Gbold$ such that $F\subset \Hbold(\RR)$. Then again by Lemma \ref{HK}, $\Hbold$ is of type $\Kscr$ and $\Hbold(\RR)^+ = F$. The property (b) implies that $\Gbold_n^{\text{der}}(\RR)^+\subset\Hbold(\RR)^+$, for $n\gg0$. Hence we deduce that
$$\Gbold_n^{\text{der}}\subset\Hbold, n\gg0.$$

For $n$ big enough, since $Z_n$ is a non-factor special subvariety of $S$, by Proposition \ref{centNF}, the centralizer $\Zcent_\Gbold(\Gbold_n^\text{der})(\RR)$ is compact. In particular the $\QQ$-subgroup $\Zcent_\Gbold(\Gbold_n^\text{der})$ is $\QQ$-anisotropic; that is, it contains no non-trivial $\QQ$-split torus. Hence $\Hbold$ is reductive by \cite{EMS97} Lemma 5.1. Since $\Hbold$ is of type $\Kscr$, it follows that $\Hbold$ is semi-simple of non-compact type.

 Let $\Dcal^+_\infty\subset\Dcal^+$ be the $\Hbold(\RR)^+$-conjugacy class of $h$, $W_\infty := \pi_h(\G\bs\G F)$ and $\mu_\infty := (\pi_h)_*\mu_F$.
 
 \begin{lemma} \label{convmeasure}
 The sequence of measures $((\pi_{h_n})_*\mu_n)_{n\in\NN}$ weakly converges to $\mu_\infty$.
 \end{lemma}
 
 \begin{proof} 
 For any continuous function $f$ on $\G\backslash\Dcal^+$ with compact support, we have
\begin{eqnarray*}
\pi_{h_n*}\mu_n(f)-\pi_{h*}\mu_F(f) &=& \mu_n(f\pi_{h_n})-\mu_F(f\pi_h)\\
                                      &=& \mu_n(f\pi_{h_n})-\mu_n(f\pi_h)+\mu_n(f\pi_h)-\mu_F(f\pi_h).
\end{eqnarray*}
By property (a), we have $\mu_n(f\pi_h)-\mu_F(f\pi_h)\to 0$, as $n\to\infty$. Since $h_n\to h$ as $n\to\infty$ by property (c), the sequence $(\pi_{h_n})_{n\in\NN}$ converges to $\pi_h$ and uniformly on all compact subsets. Since $\mu_n$ are probability measures, we have $\mu_n(f\pi_{h_n})-\mu_n(f\pi_h)\to 0$ as $n\to\infty$. Hence we have the convergence.
\end{proof}
\hfill

\textbf{Step 2. Show that $\Dcal^+_\infty$ is "horizontal".} \\

For this purpose, we recall a basic result of Griffiths on the analyticity of period images.

\begin{theorem}[Griffiths \cite{GriPerIII} Theorems 9.5 and 9.6] Let $\bar{S}$ be a smooth projective compactification of $S$ with $\bar{S}\bs S$ normal crossing divisor. Let $S^\prime$ be the union of $S$ with those points at infinity around which the monodromies are of finite order. Then the period map $\psi$ extends holomorphically to a proper map $\psi^\prime: S^\prime\to \Hod_\G^\circ(S, \VV)$ and the image $\psi^\prime(S^\prime)$ contains $\psi(S)$ as the complement of an analytic subvariety.
\end{theorem}

\begin{lemma}
$W_\infty$ is contained in $\psi^\prime(S^\prime)$.
\end{lemma}

\begin{proof}
As the period map for each special subvariety $Z_n$ of Shimura type is dominant, $\psi(Z_n)$ will necessary be analytically dense in $W_n$. Since $\psi^\prime$ is closed, we have $W_n\subset \psi^\prime(S^\prime)$, for all $n\in\NN$.  

Suppose that $W_\infty\backslash \psi^\prime(S^\prime)\neq\varnothing$. Let $x\in W_\infty\backslash \psi^\prime(S^\prime)$ and let $U_x$ be an open neighborhood of $x$ such that $U_x\cap \psi^\prime(S^\prime) = \emptyset$.
By the definition of support of measure, $\mu_\infty(U_x) > 0$. But $\pi_{h_n*}\mu_n(U_x) = 0$ for any $n\in\NN$, which contradicts to the convergence of the measures (Lemma \ref{convmeasure}).
\end{proof}

\textbf{Step 3. Show that $\Dcal_\infty^+$ admits a (unique) complex structure for which the canonical family of Hodge structures (that is, the family associated to the adjoint representation of $\Hbold^{\textrm{ad}}$ on the Lie algebra $\Lie(\Gbold)$) varies holomorphically.}\\

Fix any big enough integer $n$ such that $\Gbold^\textrm{der}_n\subset\Hbold$. Let $\Gbold_n = \Tbold_n\Gbold_n^{\text{der}}$ be the almost direct product decomposition of $\Gbold_n$, where $\Tbold_n$ is the connected center of $\Gbold_n$.

\begin{proposition} \label{normalizer0}
$\Tbold_n$ normalizes $\Hbold$.
\end{proposition}

The proof of Proposition \ref{normalizer0} follows the same strategy as the proof of \cite{Ullmo07}, Theorem 3.15 in the Shimura variety case. It contains some differences since we are working in the non-classical setting.
We shall provide all the details in the next section.
\\

Let us proceed to finish step 3.
\\

Let $\Hbold_n$ be the algebraic subgroup of $\Gbold$ generated by $\Tbold_n$ and $\Hbold$. Then $\Hbold_n$ is a reductive $\QQ$-group. Let $X_n$ be the  $\Hbold_n(\RR)$-conjugacy class of $h_n$ and $X_n^+$ be the connected component of $X_n$ containing $h_n$. Then we have 
$$h_n: \SS\to\Gbold_{n,\RR}\to\Hbold_{n,\RR}\to\Gbold_\RR.$$

Let $C = h_n(\sqrt{-1})$. It is easy to see that the Killing form is a $C$-polarization for the faithful adjoint representation of $\Hbold_{n,\RR}^\textrm{ad}$ on the Lie algebra $\Lie(\Gbold)_\RR$. So $(\Hbold_n, X_n)$ is a Hodge subdatum of $(\Gbold, \Dcal)$. In particular, $X_n^+$ admits a unique complex structure for which the canonical family of Hodge structures varies holomorphically and this complex structure on $X_n^+$ is compatible with complex structure on $\Dcal^+$.
\\

By step 2, the holomorphic tangent bundle of $X_n^+$ is contained in the holomorphic horizontal tangent bundle of $\Dcal^+$, i.e., the canonical holomorphic family of Hodge structures on $X_n^+$ is a variation of Hodge structure; that is, satisfying Griffiths transversality condition:
$$F^{-1}\Lie(\Hbold_n)_\CC = \Lie(\Hbold_n)_\CC.$$
Since the Hodge structure $\Lie(\Hbold_n)$ is of weight $0$, it must be of type
$$\{(-1, 1), (0,0), (1,-1)\}.$$
Hence the subvarieties $S_n = \pi_{h_n}(\G\bs\G\Hbold)\cong(\G\cap{\Hbold_n(\RR)}_+)\bs X_n^+$ are special subvarieties of $\Hod_\G(S,\VV)$ of Shimura type, which are also of non-factor type as each of them contains a non-factor special subvariety $W_n$ for $n\gg 0$ respectively.

We thus obtained a sequence of probability measures $(\pi_{h_n})_*\mu_F$ with support $S_n$, which obviously converges to $\mu_\infty = (\pi_h)_*\mu_F$.

\begin{lemma}
	The sequence $(S_n)$ stabilizes as $n$ tends to $\infty$. In particular, we have $W_\infty = S_n$ for any big enough $n$.
\end{lemma}

\begin{proof}
	Note that for $n\gg 0$, we have
	$$S_n\cong \Gamma\cap \Hbold(\RR)^+\bs \Hbold(\RR)^+/\Hbold(\RR)^+\cap M_n,$$
	where $M_n$ is the stabilizer of $h_n$ in $\Gbold(\RR)$. Let $K_n$ be the unique maximal compact subgroup of $\Gbold(\RR)$ containing $M_n$. Since $S_n$ are horizontal, we have
	$$\Hbold(\RR)^+\cap M_n = \Hbold(\RR)^+\cap K_n.$$
	
	Since $S_n$ are locally symmetric spaces, the $K_n\cap\Hbold(\RR)^+$ are maximal compact subgroups of $\Hbold(\RR)^+$. In particular, they all conjugate to each other by elements of $\Hbold(\RR)^+$. Fix $n_0\gg 0$. For any $n\geq n_0$, there exists a $g_n\in\Gbold(\RR)$ such that $\pi_{g_nh_{n_0}}(\G\bs\G\Hbold(\RR)^+) = \pi_{h_n}(\G\bs\G\Hbold(\RR)^+)$. And hence there exists a $v_n\in\Hbold(\RR)^+$ such that
	$$g_n(K_{n_0})g_n^{-1}\cap\Hbold(\RR)^+ = v_n^{-1}(\Hbold(\RR)^+\cap K_{n_0})v_n.$$
	So $v_ng_n$ normalizes $K_{n_0}$ and hence are in $K_{n_0}$. Write $g_n = v_n^{-1}x$ for some $x\in K_{n_0}$, we have
	\begin{equation}\label{stable}
		S_n = \G\bs\G\Hbold(\RR)^+v_n^{-1}xh_{n_0} = \G\bs\G\Hbold(\RR)^+xh_{n_0} = S_{n_0}.
	\end{equation}
	The last equality of (\ref{stable}) is because $\Hbold(\RR)^+\cap xM_{n_0}x^{-1}$ is a maximal compact subgroup of $\Hbold(\RR)^+$ and
	$$ \Hbold(\RR)^+\cap xM_{n_0}x^{-1} \subset \Hbold(\RR)^+\cap xK_{n_0}x^{-1} = \Hbold(\RR)^+\cap K_{n_0} = \Hbold(\RR)^+\cap M_{n_0}.$$
	\end{proof}

The last statement of Theorem \ref{Equidis:MT} is then clear.

\subsection{Proof of Theorem \ref{NF:finite}}
We argue by contradiction. Suppose that $S$ contains infinitely many distinct non-factor special subvarieties that are of Shimura type with dominant period maps and all are maximal among such kind of special subvarieties. Choose any sequence $(Z_n)_{n\in\NN}$ in the set of such kind of special subvarieties. Let $(W_n)_{n\in\NN}$ be the corresponding sequence of non-factor special subvariety of $\Hod^\circ_\G(S, \VV)$  and let $\mu_{W_n}$ be the canonical Borel probability measure on $\Hod^\circ_\G(S, \VV)$ with $\supp(\mu_{W_n})=W_n$. By possibly passing to a subsequence, Theorem \ref{Equidis:MT} tells us that $\mu_{W_n}$ is weakly convergent to $\mu_\infty$ and $W_n\subset W_\infty, n\gg0$ for a non-factor special subvariety $W_\infty$ of $\Hod^\circ_\G(S, \VV)$ of Shimura type. Hence $Z_n$ is contained in some positive irreducible component of $\psi^{-1}(W_\infty) $ for $k\gg0$. Since $\psi^{-1}(W_\infty)$ has only finitely many irreducible components by Theorem \ref{BKT} and the irreducible components containing some $Z_n$ are non-factor special subvarieties $S$ of Shimura type, by maximality of $Z_n$, we deduce that the there are only finitely many possibilities for $Z_n$ when $n$ is big enough, which is a contradiction.

If $(S, \VV)$ is not of Shimura type, the union of all non-factor special subvarieties of Shimura type with dominant period maps is a proper closed subvariety of $S$, which contradicts to the assumption. So $\Hod^\circ_\G(S,\VV)$ is a connected Shimura variety.

If the period map $\psi: S\to\Hod^\circ_\G(S,\VV)$ is not dominant, then by \cite{Ullmo07} Theorem 1.3, the union $U$ of non-factor special subvarities of $\Hod^\circ_\G(S,\VV)$ contained in $\overline{\psi(S)}^\text{Zar} = \psi^\prime(S^\prime)$ is a proper closed subvariety of $\overline{\psi(S)}^\text{Zar}$. Hence the preimage $\psi^{-1}(U)$ is a proper closed subvariety of $S$ containing all non-factor special subvarieties of $S$, which again contradicts to the assumption.
And we finish the proof of Theorem \ref{NF:finite}.

\section{Proof of the Proposition \ref{normalizer0}} \label{normalizer}

Fix an arbitrary $n\gg 0$ and write $\Ebold := \Gbold_n^\textrm{der}$, $\Tbold:=\Tbold_n$. We assume that $\Ebold$ is strictly contained in $\Hbold$, otherwise there is nothing to show.

\begin{lemma} \label{normalizer1}
For any $\QQ$-simple factor $\Bbold$ of $\Hbold$, there exists a noncompact $\RR$-simple factor $\Lbold_\RR$ of $\Bbold_\RR$ which is normalized by $\alpha(\UU^1)$, where $\alpha\in \Dcal_n^+$ is a Hodge generic point and $\UU^1$ is the circle subgroup of $\SS$. And if $\Bbold$ is a $\QQ$-simple factor of $\Hbold$ such that the projections of $\Ebold_\RR$ to all noncompact $\RR$-simple factors of $\Bbold_\RR$ are surjective, then $\alpha(\UU^1)$ normalizes $\Bbold_\RR$.
\end{lemma}

\begin{proof}
For the first statement, it suffices to find an element $u\in\UU^1$ of infinite order such that $\alpha(u)$ normalizes a $\RR$-simple factor $\Lbold_\RR$ of $\Bbold_\RR$.
Let $u$ be any element of $\UU^1$ of infinite order and we construct a decreasing sequence $(\Bbold_n)_{n\in\NN}$ of $\RR$-algebraic subgroups of $\Hbold$ inductively as follows:
$$\Bbold_0 = \Hbold_\RR$$
and for $n\geq 1$
$$\Bbold_n = (\Bbold_{n-1}\cap \alpha(u)\Bbold_{n-1}\alpha(u)^{-1})^\circ.$$
Note that $\Ebold_\RR \subset \Bbold_n$ for any $n\geq 0$. So the sequence $\Bbold_n$ must be stable by dimension reason. We denote the limit by $\Bbold_\infty$. By construction, the limit is normalized by $\alpha(u)$ hence also normalized by $\alpha(\UU^1)$.

Let $\Bbold$ be a $\QQ$-simple factor of $\Hbold$. Since $\Bbold$ is $\RR$-isotropic and $\Zcent_\Gbold(\Ebold)(\RR)$ is compact, the projection of $\Ebold$ to $\Bbold$ is nontrivial. Let $\Abold$ be a $\QQ$-simple factor of $\Ebold$ such that the projection of $\Abold$ to $\Bbold$ is nontrivial. Let $\Fbold_\RR$ be a noncompact $\RR$- simple factor of $\Abold_\RR$, then there exists a noncompact $\RR$-simple factor $\Lbold_\RR$ of $\Bbold_\RR$ such that the projection of $\Fbold_\RR$ to $\Lbold_\RR$ is nontrivial. Since $\alpha(\UU^1)$ normalizes $\Fbold_\RR$, the image of the projection is contained in $\Lbold_\RR\cap \Bbold_\infty$ which is thus noncompact. By the following Sublemma \ref{Sublem}, we conclude the first statement.

\begin{sublemma} \label{Sublem}
$\Lbold_\RR\cap\Bbold_\infty = \Lbold_\RR$.
\end{sublemma}
\begin{proof}
Since $\Int(\alpha(\sqrt{-1})$ is a Cartan involution of $\Gbold_\RR$ and fixes $\Ebold_\RR$ and $\Bbold_\infty$, we have Cartan decompositions:
\begin{eqnarray}
\Gbold(\RR) &=& P K,\\
\Ebold(\RR) &=& (P\cap\Ebold(\RR)) (K\cap\Ebold(\RR)), \label{Car1}\\
\Bbold_\infty(\RR) &=& (P\cap\Bbold_\infty(\RR)) (K\cap\Bbold_\infty(\RR)),
\end{eqnarray}
where $K = Z_{\Gbold(\RR)}(\alpha(\sqrt{-1}))$.
Let $M$ be the stablizer of $\alpha$ in $\Gbold(\RR)$, then we have
$$\Zcent_\Gbold(\Ebold)(\RR)\subset M\subset K.$$

By a result of Mostow \cite{Mostow55} on self-adjoint groups, for the inclusion of subgroups 
$$\Ebold_\RR\subset \Hbold_\RR\subset\Gbold_\RR$$
there exists a $g\in\Gbold(\RR)$ such that we have Cartan decompositions:
\begin{eqnarray}
\Ebold(\RR) &=& (gPg^{-1}\cap\Ebold(\RR)) (gKg^{-1}\cap\Ebold(\RR)),\label{Car2}\\
\Hbold(\RR) &=& (gPg^{-1}\cap\Hbold(\RR))(gKg^{-1}\cap\Hbold(\RR)).
\end{eqnarray}

As $\Ebold(\RR)$ admits two Cartan decompositions (\ref{Car1}) and (\ref{Car2}), they are related by an inner automorphism of $\Ebold(\RR)$, i. e., there exists $t\in\Ebold(\RR)$ such that 
\begin{eqnarray*}
gPg^{-1}\cap\Ebold(\RR) &=& t(P\cap\Ebold(\RR))t^{-1},\\
gKg^{-1}\cap\Ebold(\RR) &=& t(K\cap\Ebold(\RR))t^{-1}.
\end{eqnarray*}
Let $\g:=t^{-1}g$, then we have
\begin{eqnarray}
\g P\g^{-1}\cap\Ebold(\RR) &=& P\cap\Ebold(\RR),\label{Comp1}\\
\g K\g^{-1}\cap\Ebold(\RR) &=& K\cap\Ebold(\RR).\label{Comp2}
\end{eqnarray}

Write $\g = pk$ with $p\in P$ and $k\in K$. For any $p_1\in P\cap\Ebold(\RR)$, there exists a $p_2 \in P$ such that $p_2 = \g^{-1}p_1\g$. So $p^{-1}p_1p = kp_1k^{-1}\in P$, which implies that $p^2p_1 = p_1p^2$, i.e.
$$p^2\in Z_{\Gbold(\RR)}(P\cap\Ebold(\RR)).$$
Similarly, we can show that $$p^2\in Z_{\Gbold(\RR)}(K\cap\Ebold(\RR)).$$ So $p^2\in Z_{\Gbold(\RR)}(\Ebold(\RR))\subset K$ which implies $p = 1$ and $\g\in K$. Hence we have
$$gKg^{-1}\cap\Hbold(\RR) = tKt^{-1}\cap\Hbold(\RR) = t(K\cap\Hbold(\RR))t^{-1}.$$ 
And thus we have Cartan decompositions of $\Hbold(\RR)$ and $\Lbold_\RR(\RR)$:
\begin{eqnarray*}
\Hbold(\RR) &=& (P\cap\Hbold(\RR))(K\cap\Hbold(\RR)),\\
\Lbold_\RR(\RR) & = & (P\cap\Lbold_\RR(\RR))(K\cap\Lbold_\RR(\RR)).
\end{eqnarray*}

Since $\Dcal^+_\infty$ is "horizontal" as showed by Step 2 in the proof of Theorem \ref{Equidis:MT}, we can deduce that $K\cap\Hbold(\RR) = M\cap\Hbold(\RR)$. In particular,
$$\alpha(u)(K\cap\Hbold(\RR))\alpha(u)^{-1} = K\cap\Hbold(\RR),$$
which implies that 
$$K\cap\Bbold_\infty(\RR) = K\cap\Hbold(\RR).$$

Therefore 
$$K\cap\Lbold_\RR(\RR)\subset\Bbold_\infty(\RR)\cap\Lbold_\RR(\RR)\subset\Lbold_\RR(\RR).$$

Since $\Lbold_\RR(\RR)$ is simple and noncompact, the subgroup $K\cap\Lbold_\RR(\RR)$ is a maximal proper closed subgroup of $\Lbold_\RR(\RR)$. Hence we have 
$$\Bbold_\infty(\RR)\cap\Lbold_\RR(\RR)=\Lbold_\RR(\RR).$$
This finishes the proof of the sublemma.
\end{proof}

Now let us prove the second statement of Lemma \ref{normalizer1}. Let $\Bbold$ be a $\QQ$-simple factor of $\Hbold$ such that the projections of $\Ebold_\RR$ to all noncompact $\RR$-simple factors of $\Bbold_\RR$ are surjective.
If $\Lbold_\RR$ is a compact $\RR$-simple factor of $\Bbold_\RR$, then $\Lbold_\RR = K\cap\Lbold_\RR = M\cap\Lbold_\RR$ and hence $\Lbold_\RR$ is normalized by $\alpha(\UU^1)$. If $\Lbold_\RR$ is a noncompact $\RR$-simple factor of $\Bbold_\RR$, by assumption the projection of $\Ebold$ to $\Lbold_\RR$ is surjective, from which we deduce that $\Bbold_\infty(\RR)\cap\Lbold_\RR(\RR)$ is noncompact, and thus equals to $\Lbold_\RR(\RR)$. So $\Lbold_\RR$ is again normalized by $\alpha(\UU^1)$. Therefore $\Bbold_\RR$ is normalized by $\alpha(\UU^1)$ and we finish the proof of Lemma \ref{normalizer1}.
\end{proof}

Let $\Sscr$ be the poset
$$\Sscr = \{\Fbold\subset\Gbold | \Fbold~\text{is a semisimple}~\QQ\text{-subgroup of type}~\Kscr~\text{and}~ \Ebold\subsetneq\Fbold\subset\Hbold\}$$
with the partial oder given by inclusion.

\begin{lemma}
In order to prove Proposition \ref{normalizer0}, it suffices to assume that $\Hbold$ is a minimal element of $\Sscr$.
\end{lemma}

\begin{proof}
Let $\Fbold$ be a minimal element of $\Sscr$. By assumption, $\Tbold$ normalizes $\Fbold$. We have an almost direct product decomposition $\Tbold = (\Tbold\cap\Fbold)\Tbold^\prime$ with $\Tbold^\prime$ centralizes $\Fbold$. The algebraic subgroup $\Fbold^\prime$ generated by $\Fbold$ and $\Tbold$ is reductive and has an almost direct product decomposition $\Fbold^\prime = \Tbold^\prime\Fbold$ $(\Fbold = \Fbold^{\prime\textrm{der}})$. Let $\Dcal^\prime$ be the $\Fbold^\prime(\RR)$-conjugacy class of $\alpha$. Note that $\Dcal^{\prime +}$ is automatically "horizontal" as it is contained in $\Dcal_\infty^+$. Hence by the same reasoning as in the last part of the proof of Theorem \ref{Equidis:MT}, $(\Fbold^\prime, \Dcal^\prime)$ is a Shimura datum. It is easy to see that $\Fbold^\prime$ is the generic Mumford-Tate group of $\Dcal^\prime$. Let $\alpha^\prime$ be a Hodge generic point of $\Dcal_n^\prime$ and replace the Shimura datum $(\Gbold, \Dcal)$ by $(\Fbold^\prime, \Dcal^\prime)$, $\Tbold$ by $\Tbold^\prime$ and $\alpha$ by $\alpha^\prime$. We thus reduce the proof of Proposition \ref{normalizer0} to the inclusion $\Fbold\subset\Hbold$. And after iterating the above procedure finitely many times, the Proposition \ref{normalizer0} follows.
\end{proof}

We suppose now $\Hbold$ is minimal in the set $\Sscr$ and proceed to finish the proof of the Proposition \ref{normalizer0}.
\\

Let $\Tbold^\prime = \Tbold\cap\Nor_\Gbold(\Hbold)^\circ$ and write $\Tbold = \Tbold^\prime\Tbold^{\prime\prime}$ as an almost direct product. Suppose that $\Tbold^{\prime\prime}$ is nontrivial. Note that $\alpha(\UU^1) = \alpha(\SS)$ is not contained in $\Tbold^\prime(\RR)\Ebold(\RR)$ as $\Gbold$ is adjoint and $\Gbold_n$ is the generic Mumford-Tate group of $\Dcal^+_n$. We can choose $b = ag\in \alpha(\UU^1)$ with $g\in\Ebold(\RR)$ and $a = a^\prime a^{\prime\prime}\in \Tbold(\RR)$ such that $a^\prime\in\Tbold^\prime(\RR), a^{\prime\prime}\in\Tbold^{\prime\prime}(\RR)$ and $a^{\prime\prime}\notin\Tbold^\prime(\RR)\cap\Tbold^{\prime\prime}(\RR)$.\\

Since $\Tbold^\prime(\QQ)$ (resp. $\Tbold^{\prime\prime}(\QQ)$) is dense in $\Tbold^\prime(\RR)$
 (resp.$\Tbold^{\prime\prime}(\RR)$) for the usual topology, we can find a sequence $(a_n = a^\prime_na^{\prime\prime}_n\in\Tbold(\QQ))_{n\in\NN}$ such that $a^\prime_n\in\Tbold^\prime(\QQ)$ (resp. $a^{\prime\prime}_n\in\Tbold^{\prime\prime}(\QQ))$ converges to $a^\prime$ (resp. $a^{\prime\prime}$). We can also assume that $a_n^{\prime\prime}\notin\Tbold^\prime(\RR)\cap\Tbold^{\prime\prime}(\RR)$,for all $n\in\NN$.\\

Now consider $\Hbold_n^\prime:=(a_n\Hbold a_n^{-1}\cap\Hbold)^\circ$. As $\Hbold_n^\prime$ contains $\Ebold$, it is reductive by \cite{EMS97} Lemma 5.1. We have an almost direct product decomposition 
$$\Hbold_n^\prime = \Hbold_{n}^{\text{nc}}~\Hbold^{\prime\prime}_n$$
where $\Hbold_{n}^{\text{nc}}$ is the almost direct product of $\QQ$-simple factors of noncompact type of $\Hbold_n^\prime$ and $\Hbold^{\prime\prime}_n$ is the almost direct product of the remaining factors, so in particular $\Hbold_{n}^{\text{nc}}\in\Sscr$. Note that the projection of $\Ebold$ to $\Hbold^{\prime\prime}_n$ is trivial, so
$$\Hbold^{\prime\prime}_n(\RR)\subset\Zcent_\Gbold(\Ebold)(\RR)\subset M\cap\Hbold(\RR)=K\cap\Hbold(\RR).$$
Therefore
$$\Ebold\subset\Hbold_{n}^{\text{nc}}\subset\Hbold.$$
By minimality of $\Hbold$ and $a_n^{\prime\prime}\notin\Tbold^\prime(\QQ)$, we have 
$\Hbold_{n}^{\text{nc}}=\Ebold$ and
$$\Hbold_n^\prime = \Ebold~\Hbold^{\prime\prime}_n.$$
Let $n$ tend to infinity, we deduce that every element of $(a\Hbold(\RR)a^{-1}\cap\Hbold(\RR))^+$ can be written as a product of an elements of $\Ebold(\RR)$ and an element of $M\cap\Hbold(\RR)$.\\

Let $\Bbold$ be a $\QQ$-simple factor of $\Hbold$, by Lemma \ref{normalizer1}, there exists a noncompact $\RR$-simple factor $\Lbold_\RR$ of $\Bbold_\RR$ which is normalized by $\alpha(\UU^1)$. Since $b\in \alpha(\UU^1)$ and $a\Hbold(\RR)a^{-1}\cap\Hbold(\RR) = b\Hbold(\RR)b^{-1}\cap\Hbold(\RR)$, we have
$$\Lbold_\RR(\RR)^+\subset (a\Hbold(\RR)a^{-1}\cap\Hbold(\RR))^+,$$
which implies that $\Ebold_\RR$ projects surjectively onto $\Lbold_\RR$. In particular, $\Lbold_\RR$ is also a $\RR$-simple factor of $\Ebold_\RR$.

\begin{sublemma} \label{MTsimple}
The smallest $\QQ$-algebraic subgroup $\Fbold$ of $\Gbold$ which contains $\Lbold_\RR$ is $\Bbold$.
\end{sublemma}

\begin{proof}
Let $h\in\Bbold(\QQ)$. Then $\Lbold_\RR\subset h\Fbold_\RR h^{-1}$. Since $h\Fbold h^{-1}$ is a $\QQ$-subgroup of $\Gbold$, we have $\Fbold\subset h\Fbold_\RR h^{-1}$. So $\Fbold$ is a normal subgroup of $\Bbold$ and hence equals to $\Bbold$ as $\Bbold$ is $\QQ$-simple. 
\end{proof}

By Sublemma \ref{MTsimple}, $\Bbold$ is contained in $\Ebold$. In particular, the projection of $\Ebold_\RR$ to any noncompact $\RR$-simple factor of $\Bbold_\RR$ is surjective. By Lemma \ref{normalizer1}, $\alpha(\UU^1)$ normalizes $\Bbold_\RR$ and hence normalizes $\Hbold_\RR$, i. e.,
$$\alpha(\UU^1)\subset\Nor_\Gbold(\Hbold)_\RR.$$
And thererfore $$\Gbold_n\subset\Nor_\Gbold(\Hbold).$$

In particular, $\Tbold_n\subset\Nor_\Gbold(\Hbold)$ which contradicts to our assumption that $\Tbold_n^{\prime\prime}$ is nontrivial. This completes the proof of Proposition \ref{normalizer0}.


\end{document}